\documentclass[11pt,1eqno]{article}
   \usepackage{amsthm}
   \usepackage{amsmath}
   \usepackage{amssymb}
   \usepackage{amsfonts}
   \usepackage{amscd}
   \usepackage[mathscr]{eucal}
   \usepackage{color}
   \usepackage{verbatim} \usepackage{graphics} \usepackage{hyperref}

\newcommand{\wt}{\widetilde}

\newcommand{\p}{\pmod{p}}

\newcommand{\M}{\pmod{M}}

\newcommand{\zp}{\mathbb{Z}/p\mathbb{Z}}
\newcommand{\zP}{\mathbb{Z}/P\mathbb{Z}}

\newcommand{\Card}{\operatorname{Card}}

\newcommand{\card}{\operatorname{card}}

\newtheorem{thm}{Theorem}[section]
\newtheorem*{mlem*}{Main Lemma}
\newtheorem*{thm*}{Theorem}
\newtheorem{prop}[thm]{Proposition}
\newcommand{\N}{{\mathbb N}}
\newcommand{\R}{{\mathbb R}}
\newcommand{\Z}{{\mathbb Z}}
\newcommand{\C}{{\mathbb C}}
\newtheorem{lem}[thm]{Lemma}
\newcommand{\E}{{\mathbb E}}
\newtheorem*{cor*}{Corollary}
\newtheorem*{conj*}{Conjecture} 
\newtheorem{cor}[thm]{Corollary}

\theoremstyle{definition}
\newtheorem{defi}[thm]{Definition}

\newtheorem{rem}[thm]{Remark}

\begin{document}

\title{A remark on a result of Helfgott, Roton and Naslund}
\author{ Gyan Prakash\\  Harish-Chandra Research Institute\\ 
Chhatnag Road, Jhunsi\\ Allahabad 211019, India.\\ E-mail:
\texttt{gyan@hri.res.in}  
} 
\date{} 
\maketitle
\begin{abstract}
Let $F(X)= \prod_{i=1}^k(a_iX+b_i)$ be a polynomial with $a_i, b_i$ being integers. Suppose the discriminant of $F$ is non-zero and $F$ is admissible. 
Given any natural number $N$, let $S(F,N)$ denotes those integers less than or equal to $N$ such that $F(n)$ has no prime factors less than or equal to $N^{1/(4k+1)}.$  Let $L$ be a translation invariant linear equation in $3$ variables. Then any $A\subset S(F, N)$ with $\delta_F(N): = \frac{\card(A)}{\card(S(F,N))} \gg_{\epsilon, F, L}\frac{1}{(\log \log N)^{1-\epsilon}}$ contains a non-trivial solution of $L$ provided $N$ is sufficiently large.
\end{abstract}
Given $A\subset \N$ and a natural number $N$ we set $A(N): = A \cap [1,N].$ Given any natural number $k$  we write $\log_kN: = \underbrace{\log....\log}_{\text{$k$-times}} N.$
Given a subset $P_1$ of the set of primes $P$, we define the relative density $\delta_P(N) = \frac{\card(P_1(N))}{\card(P(N))}.$ In~\cite{BJG}, Ben Green showed that any subset $P_1$ of the set of primes $P$, with relative density $\delta_P(N)) \geq c (\log_5 N/\log_3 N)^{1/2}$ for some $N\geq N_0$, where $c$ and $N_0$ are absolute constants, contains a non-trivial $3$-term arithmetic progression. In~\cite{Harald}, H. Helfgott and A. De. Roton improved this result to show that the same conclusion holds under the weaker assumption that the relative density $\delta_P(N) \geq c (\log_3  N/\log_2  N)^{1/3}.$ In~\cite{Eric}, Eric Naslund, using a modification of the arguments of Helfgott and Roton, showed that the result holds under even a weaker assumption $\delta_P(N) \geq c(\epsilon) (1/\log_2  N)^{1-\epsilon}$, where $\epsilon >0$ is any real number and $c(\epsilon)>0$ is a constant depending only on $\epsilon.$ The purpose of this note is to observe that the arguments of Helfgott and Roton~\cite{Harald} and Eric Naslund~\cite{Eric} gives a more general result, namely Theorem~\ref{mainthm} stated below.

\vspace{0.5cm}
\noindent
Let $F(X)= \prod_{i=1}^{k}(a_iX+b_i) \in \Z[X]$ with $a_i \in \Z\setminus \{0\}$ and $b_i \in \Z$. Moreover we suppose that
\begin{enumerate}
\item the discriminant of $F$; $\Delta(F) = \prod_{i=1}^ka_i\prod_{i\ne j}(a_ib_j-b_ia_j) \ne 0$ and
\item is admissible that is to say that for all primes $p$, there exists $n \in \Z$ such that $F(n)\not\equiv 0 \p.$
\end{enumerate}
Hardy-Littlewood conjecture predicts that for any $F$ as above, the number of integers $n \leq N$ such that $a_in +b_i$ is prime for all $i$ is asymptotically equal to $c(F)\frac{N}{\log^kN}$. This is not known except the case when $F$ is a linear polynomial. However using Brun's sieve we know a lower bound for the number of $n \leq N$ such that the number of prime factors of $a_in + b_i$ is at most $4k+1$. Given any real number $z>0$, let $P(z) = \prod_{p\leq z}p$ and we set
\begin{equation}
S_F(N,z)= \{n \leq N: {\text gcd}\left(F(n), P(z)\right)=1\}.
\end{equation}
Then using Brun's sieve~\cite[see page number 78, (6.107)]{FI}, we know the following lower bound
\begin{equation}\label{Brun}
\Card\left(S_F(N,N^{1/(4k+1)})\right) \geq c_1(F)\frac{N}{\log^kN},
\end{equation}
 where $c_1(F)>0$ is a constant depending only upon $F.$ Given any $A\subset S_F(N,N^{1/(4k+1)})$ we define the {\it relative density} $\delta_F(N)$ of $A$ to be
\begin{equation}\label{density}
\delta_{F}(N) = \frac{\Card(A(N))}{\Card\left(S_F(N,N^{1/(4k+1)})\right)}.
\end{equation} 
For the brevity of notation, we shall also write $\delta(N)$ or simply $\delta$ to denote $\delta_F(N).$

\vspace{2mm}
\noindent
Let $s\geq 3$ be a natural number and 

\begin{equation}\label{dlinear}
L: = c_1x_1+\cdots+c_sx_s=0
\end{equation}
be a linear equation with $c_i \in \Z\setminus \{0\}.$  The linear equation $L$  is  said to be {\it translation invariant} if $\sum_i c_i = 0.$ A solution $(x_1,\cdots,x_s)$ of $L$ is said to be {\it non-trivial} if for some $i,j$ we have $x_i \ne x_j.$

\begin{thm}\label{mainthm}
Let $k$ be a natural number, $F(X) = \prod_{i=1}^k(a_iX+b_i)\in \Z[X]$ with $a_i\in \Z \setminus \{0\}$ and $b_i\in \Z$  for all $i$. Suppose that $F$ is admissible and the discriminant of $F$ is non-zero. Let $L$ be a translation invariant linear equation as defined in~\eqref{dlinear} with $s=3.$ Then given any $\epsilon > 0$ there exists a constant $c(F,L,\epsilon) >0$ and a natural number $N(F,L,\epsilon)$ such that the following holds. Given any $N\geq N(F,L,\epsilon)$, any set $A\subset S_F(N,N^{1/(4k+1)})$, with 
\begin{equation}\label{ldensity}
\delta_F(N)\geq c(F,L,\epsilon)\frac{1}{ (\log \log N)^{1-\epsilon}},
\end{equation}
contains a non-trivial solution of $L.$

\end{thm}

\begin{rem}
\begin{enumerate}
\item Let $P$ be the set of primes. In the above theorem, taking $F(n)=n$, $L: = x_1+x_2-2x_3=0$, and $A$ to be a subset of
  primes $P(N)$, with $\Card(A) \geq \delta \Card(P(N))$, with $\delta$
  satisfying~\eqref{ldensity},  one recovers the result of Eric
  Naslund~\cite{Eric} stated above.

\item We say that a prime $p$ is a Chen prime if $p+2$ has at most two
  prime factors and also any prime factor of $p+2$ is greater than
  $p^{1/10}.$ Green and Tao~\cite[Theorem 1.2]{GT-enveloping} had shown that Chen primes contain a non
  trivial 3-term arithmetic progression. In  the above theorem taking $F(X)
  = X(X+2)$, $L: = x_1+x_2-2x_3=0$, we obtain any subset $A$ of Chen primes with relative
  density $\delta$, with $\delta$ satisfying~\eqref{ldensity}, contains
  a non-trivial 3-term arithmetic progression.
\end{enumerate}
\end{rem}

\begin{defi}\label{glhl}
 Let $L$ be a translation invariant linear equation in $s$ variables as in~\eqref{dlinear}.

\begin{enumerate}
\item Let $h_L: (0,1)\to \mathbb{R}$ be a non-negative function satisfying the following. Given any prime $P$ and a set $A\subset \Z/P\Z$ with $\Card(A) \geq \eta P$, the number of solutions of $L$ in $A$ is at-least $h_L(\eta)P^{s-1}.$
\item Let $g_L: \N\to \mathbb{R}_0^{+}$ be a monotonically decreasing function with $\lim_{N\to \infty}g_L(N)=0$ and satisfying the following properties. There exists a natural number $N_g$ such that given any $N\geq N_g$, any set $A \subset [1,N]$ with $|A| \geq g_L(N)N$ contains a non-trivial solution of $L$. 
 Given $\eta > 0$, let $g^{*-1}_L(\eta)$ denotes the smallest natural number $m$ such that $g_L(m) \leq \eta.$

\end{enumerate}
\end{defi}
\begin{rem}\label{rglhls}\begin{enumerate}
\item Let $g_L$ be a function as in Definition~\ref{glhl}. When the number of variables $s$ in $L$ is equal to $3$, then using an arguments due to Varnavides, it can be shown that the function $h_L(\eta) = \frac{\eta}{\left(2g^{*-1}_L(\eta/2)\right)^2}$ is a function satisfying the properties as in Definition~\ref{glhl}~(i). If we can find a similar relation between the function $g^*_L$ and $h_L$ for $s> 3$, then the result of Theorem~\ref{mainthm} can be extended for $s> 3$ using the following result of Thomas Bloom.
\item In~\cite{Bloom}, Thomas Bloom showed that there exists an absolute constant $c>0$ depending only on $L$ such that the function $g_L(N)= c\left(\frac{\log^5\log N}{\log N}\right)^{s-2}$ satisfies the above properties. In this case $g^{*-1}_L(\eta) \leq \exp(c_1\eta^{-1/(s-2)}\log^6 \log (\frac{1}{\eta}))$ with $c_1>0$ being a constant depending only upon $L.$
\end{enumerate}
\end{rem}

Let $z= \frac{\log N}{3}$ and $M = \prod_{p\leq z}p$. For any $b\in
\{0,1,\cdots M-1\}$, we set $$A_{b} = \{n : n \in A, n\equiv b \M \}$$
We notice that
$$
A_b \subset \{n\leq N/M: {\text gcd}\left(F(b+nM), P(N^{1/(4k+1)})\right)=1\}.$$
The following lemma is an easy consequence of $W$-trick due to Ben Green.
\begin{lem}[W-trick]\label{W-trick}
There exists a $b_0 \in \{0,1,\cdots,M-1\}$ such that ${\text gcd}(F(b_0), M)=1$ and 
$$\Card(A_{b_0}) \geq c(F)\frac{\delta \log^k\log N}{\log^k N} \frac{N}{M},  $$
where $c(F)>0$ is a constant depending only upon $F$ and $\delta$ is as in~\eqref{density}.

\end{lem}
\begin{proof}
Since $z \leq N^{1/(4k+1)}$, it follows that if $A_{b}\neq
\emptyset$, then $F(b) \not\equiv 0 ({\textrm{mod }}p)$ for all $p
\leq z.$ Now for $p$ which does not divide $ \Delta(F)\prod_{i=1}^ka_i$, the number of  solutions $n \in \zp$ of the
equation $F(n) \equiv 0 ({\textrm{mod }}p)$ is equal to $k.$  Let $\Delta'(F) =\Delta(F)\prod_{i=1}^ka_i.$ Then using Chinese
remainder theorem, it follows that the number of $b \in \{0,1,\cdots,M-1\}$ such
that  $A_{b}$ is not an empty set is at most $\frac{\prod_{p
    \leq z}(p-k)\prod_{p|\Delta'(F)}p}{\prod_{p|\Delta'(F)}(p-k)}.$ Using this,
the identity  $$\sum_{b=0}^{M-1}\card(A_{b})= \card(A)$$ and
 \eqref{Brun}, it follows that there exists a $b_0$ such that
$$\Card(A_{b_0})\geq c(F)\delta \prod_{p\leq z}(p-k)^{-1}\frac{N}{\log^kN}\, =\, c(F)\delta \prod_{p\leq z}\left(1-\frac{k}{p}\right)^{-1} \frac{N}{M\log^kN},$$
where $c(F)=c_1(F) \prod_{p|\Delta'(F)}\frac{p-k}{p}$ with $c_1(F)$ as in \eqref{Brun}. 
The lemma follows using this,  and Mertens formula.
\end{proof}

\noindent
Let $b_0 \in \{0,1,\cdots,M-1\}$ be as provided by Lemma~\ref{W-trick}. Without any loss of generality, we may assume that $c_1,\cdots c_r>0$ and $c_{r+1},\cdots,c_s <0$. Let $$c= c_1+\cdots+c_r.$$
Let $P \in [cN/M, 2cN/M]$ be a prime and $A'$ denote the image of
$A_{b_0}$ in $\Z/P\Z$ under the natural projection map.  The set $A_{b_0}$ contains a non-trivial
solution of $L$ if and only if $A'$ contains a non-trivial solution of $L$. We shall prove
 Theorem~\ref{mainthm} by showing that $A'$ contains a non trivial solution.

\vspace{2mm}
\noindent
For any set $C\subset \Z/P\Z$, we set $d(C)= \frac{\card(C)}{P}$ to denote the density of $C$ in $\Z/P\Z$.
Given any set $C\subset \Z/P\Z$, let $f_C: \Z/P\Z\to \mathbb{R}_0^+$ be the function defined as $f_C(n) = \frac{1}{d_C}I_C(n).$
For any function $f: \Z/P\Z\to \mathbb{C}$, we set $\E(f): =\frac{1}{P} \sum_{n\in\Z/P\Z}f(n).$ Then we may verify that for any set $C$, we have $\E(f_C)=1.$
Given any integer $l\geq 1$, we write $\|f\|_l: =\left(\E(|f|^l)\right)^{1/l}.$
 
\vspace{2mm}
\noindent
The Fourier transform of $f$ is a function $\widehat{f}: \Z/P\Z\to \mathbb{C}$ 
defined as $\widehat{f}(t)= \E(f(y)\exp{(2\pi i ty)}.$
We also set
 $$\Lambda_L(f): = \sum_{n_1,\cdots,n_s\in \Z/P\Z, \sum_{i}c_in_i=0}\prod_{i=1}^sf(n_i).$$
The following identity is easy to verify:

$$\Lambda_L(f)=  
P^{s-1}\sum_{t\in\Z/P\Z}\prod_{i=1}^s\widehat{f}(c_it).$$

\vspace{2mm}
\noindent
Let $G$ be a finite commutative group. Given functions $f, g: G\to \C$,
we define the convolution function $f*g:G \to \C$ as follows:
\begin{equation}
f*g(n) = \frac{1}{|G|}\sum_{y\in G}f(n-y)g(y).
\end{equation}

\begin{prop}\label{prop-3-conv}
Let $A'\subset \Z/P\Z$ be as above and suppose $\delta > \log^{-100}P.$  Let $B\subset [-\frac{P}{4},\frac{P}{4}]$ with $\card(B)\geq \log^{k+101}P$, then given any  integer $l\geq 2$, we have
\begin{equation}
\Lambda_L(f_{A'}*f_B) \geq c_1h_L(c_2\delta^{\frac{l}{l-1}})P^{s-1},
\end{equation}
where $\delta$ is as defined in~\eqref{density} and $c_1, c_2>0$ are constant depending only upon $F$, $l$ and the linear equation $L.$ 
\end{prop}
\begin{prop}\label{restriction-cor}
Let $\epsilon_1, \epsilon_2>0$ be real numbers. Let $A'\subset \Z/P\Z$ be as above and let $S_{\epsilon_1}\subset \Z/P\Z$ be the set defined as $S_{\epsilon_1}={\rm Spec}_{\epsilon_1}(f_{A'})= \{t\in \Z/P\Z: |\widehat{f_{A'}}(t)|> \epsilon_1\}$. Let $B\subset \Z/P\Z$  such that for every $t \in S=\bigcup_ic_ic_1^{-1}.S_{\epsilon_1}$, we have $|\widehat{f_B}(t)-1| \leq \epsilon_2$, then we have
$$|\Lambda_L(f_{A'}) - \Lambda_L(f_{A'}*f_B)| \leq c(F)\frac{\epsilon_2 + \epsilon_1^{0.5}}{\delta^6}P^{s-1},$$
where $\delta$ is as defined in~\eqref{density} and $c(F)>0$ is a constant depending only upon $F$.
\end{prop}

\noindent
Let $G(X)= F(b+XM)$ be the polynomial with integer coefficients and let $S\subset \mathbb{Q}$ be the set of roots of $G.$
For proving Proposition~\ref{prop-3-conv}, we shall use the following result, which we prove using beta sieve.

\begin{prop}\label{intersec}
Let $h_1,h_2,\cdots,h_r$ be distinct integers with $|h_i| \leq N^{100}\, \, \forall i.$ Moreover suppose for $i\ne j$, we have $h_i-h_j \notin (S-S)\cap \Z$, where $S$ is the set of roots of the polynomial $G(X)= F(b+XM).$ Then we have
\begin{equation}\label{eq-intersec}
\Card\left( (A_{b_0}+h_1)\cap \cdots \cap (A_{b_0}+h_r)\right)\leq c(F,r)
\frac{N\log^{kr}z}{M\log^{kr}N},
\end{equation}
where $c(F,r)>0$ is a constant depending only upon $F$ and $r$, and in particular does not depend upon  $h_i'$s.
\end{prop}

\section{Proof of Proposition~\ref{prop-3-conv}}
In this section, we shall prove Proposition~\ref{prop-3-conv} using Proposition~\ref{intersec}.

\vspace{2mm}
\noindent
Given any $f: \Z/P\Z\to \mathbb{R}^+$, let $D(f)$ be the subset of $\Z/P\Z$ defined by $D(f):= \{n\in \Z/P\Z: f(n)> 1/2\}$. The following two lemmas are easy to verify.
\begin{lem}\label{avD}
Let $f:\Z/P\Z\to \mathbb{R}^+$ be a function with $\E(f)=1$. Then we have $\frac{1}{P}\sum_{n \in D(f)}f(n)\geq \frac{1}{2}.$ 
\end{lem} 
\begin{proof}
The result follows by observing that $\E(f) = \frac{1}{P}\sum_{n\notin D(f)}f(n) + \frac{1}{P}\sum_{n\in D(f)}f(n)$ and $\frac{1}{P}\sum_{n\notin D(f)}f(n) \leq \frac{1}{P}\sum_{n\in \Z/P\Z}\frac{1}{2}\leq \frac{1}{2}.$
\end{proof}
\begin{lem}\label{lblambdaL}
For any $f: \Z/P\Z \to \mathbb{R}^+$ with $\Card(D(f))\geq \eta P$, we have
$$\Lambda_L(f) \geq \frac{1}{2^s}h_L(\eta)P^{s-1}.$$
\end{lem}
\vspace{2mm}
\noindent
For this we need the following result which follows using the arguments from~\cite{Harald} and~\cite{Eric}.
\begin{thm}\label{3-term-modp}
Let $C\subset \Z/P\Z$ be a set with the following properties. There exists a subset $S'\subset \Z/P\Z$ with 
$S'=-S'$, $0\in S'$ and $\card(S')\leq t$ such that given any integer $l \geq 2$ and $h_1,\cdots h_l~\in~\Z/P\Z$ with $h_i-h_j\notin S'$ for $i\ne j,$
 we have

\begin{equation}\label{intersecp}
d\left((C+h_1)\cap \cdots \cap (C+h_l)\right)\leq \frac{c(l)}{\beta^l}d(C)^l,
\end{equation}
for some $\beta \leq 1$ and where $c(l) >0$ is a constant depending only upon $l$. Then for any $B\subset \Z/P\Z$ with $\card(B)\geq \frac{1}{d(C)}$, we have
\begin{enumerate}
\item the cardinality of the set $D:= \{n \in \Z/P\Z: f_C*f_B\geq 1/2\}$ is at least $c\beta^{l/(l-1)}P$ and
\item
and \begin{equation}
\Lambda_L(f_C*f_B) \geq \frac{1}{2^s}h_L\left(c\beta^{l/(l-1)}\right)P^{s-1},
\end{equation}
\end{enumerate}
where $c > 0$ is a constant depending only upon $t$ and $l.$
\end{thm}

\vspace{2mm}
\noindent
First we prove Proposition~\ref{prop-3-conv} using Theorem~\ref{3-term-modp} and Proposition~\ref{intersec}.
\begin{proof}[Proof of Proposition~\ref{prop-3-conv}] Let $S\subset \mathbb{Q}$ be the set of roots of the polynomial $G(X)= F(b+MX)\in \Z[X]$ as in Proposition~\ref{intersec} and $\pi:\Z \to \zP$ be the natural projection map. We shall prove the proposition by showing that the assumptions of Theorem~\ref{3-term-modp} are satisfied with $C=A'$ and $S'  = \pi((S-S)\cap \Z)$ and $t=k^2$ and $\beta= \delta.$

\vspace{1mm}
\noindent
Since $\card(B)\geq (\log P)^{k+101}$ and $\delta\geq \frac{1}{(\log P)^{100}P}$, using Lemma~\ref{W-trick}, it follows that we have $\card(B)d(C)\geq 1.$

\vspace{1mm}
\noindent
We have $S'=-S'$, $0\in S'$ and $$\card(S') \leq \card(S-S) \leq \card(S)^2 \leq k^2.$$
 Let $h_1,\ldots,h_l\in \zP$ be such that for $i \ne j$, we have $h_i-h_j \notin S'$. 
The result follows by showing that~\eqref{intersecp} holds with $\beta = \delta$, where $\delta$ is as in~\eqref{density}. Given $x \in \zP$, let $\wt{x}$ be the integer in $[0,P)$ with $\pi(\wt{x})=x.$ By re-ordering $h_i's$, if necessary, we may assume that $\wt{h_1}> \wt{h_2}>\ldots>\wt{h_l}.$ 

\vspace{1mm}
\noindent
Given any $n \in \cap_i (C+h_i)$, it follows that 
$\wt{n-h_i}\in A_{b_o}$ for all~$i.$ Now we observe a relation between $\wt{n-h_i}$ and $\wt{n-h_1}.$ For this note that for any $i$, we have $\wt{n-h_1}+ \wt{h_1}-\wt{h_i} \in [0,2p).$  If
$\wt{n-h_1}+ \wt{h_1}-\wt{h_i} \in [0,p)$ then we have $\wt{n-h_i} = \wt{n-h_1}+ \wt{h_1}-\wt{h_i}$
 and if $\wt{n-h_1}+ \wt{h_1}-\wt{h_i} \in [P,2P)$, then we have
$\wt{n-h_i}= \wt{n-h_1}+ \wt{h_1}-\wt{h_i}-P.$ Using this it follows there exists $j \leq l$ such that
$$\wt{n-h_1} \in A^j,$$
where $A^{j} = \cap_{i=1}^{j}(A_{b_0} + \wt{h_i}-\wt{h_1})\cap_{i=j+1}^{r}(A_{b_0}+\wt{h_i}-\wt{h_1}+P).$ Therefore it follows that 
$$\Card(\cap_i(C+h_i))\leq \sum_{j=1}^r \Card(A^j).$$
Since the condition $h_i-h_j \notin S'$ implies that for any $m \in \Z$, we have $\wt{h_i}-\wt{h_j}+mP \notin S-S$, using Proposition~\ref{intersec}, it follows that for any $j$, we have
$\Card(A^j) \leq c(F,r)
\frac{N\log^{kr}z}{M\log^{kr}N}$ and hence
\begin{equation}\label{ubdci}
d\left(\cap_i(C+h_i)\right) \leq c(F,r)
\frac{\log^{kr}\log N}{\log^{kr}N}.
\end{equation}
Since $\Card(C)= \Card(A_{b_0})$, using Lemma~\ref{W-trick}, we have
\begin{equation}\label{lbdc}
\frac{\log^k\log N}{\log^kN} \leq \frac{c(F) d(C)}{\delta}.
\end{equation}
Therefore using~\eqref{ubdci} and \eqref{lbdc}, it follows that~\eqref{intersecp} holds with $\beta = \delta.$
Hence the result follows.

\end{proof}
Now we shall prove Theorem~\ref{3-term-modp}.
For this we use the following observation from~\cite{Eric}.
\begin{prop}\label{lbD}
Let $f: \Z/P\Z\to \mathbb{R}^+$ be a non negative real valued function with $\E(f)=1$. Then if $\|f\|_l \leq \frac{c}{\beta}$ for some integer $l \geq 2$, then we have
$$\Card(D(f))\geq (c^{-1}\beta)^{l/(l-1)}P.$$
\end{prop}
\begin{proof}
Using Lemma~\ref{avD}, we have
$$\frac{1}{P}\sum_{n\in D(f)}f(n) \geq 1/2.$$
Moreover we have using H\"olders inequality
$$\frac{1}{P}\sum_{n\in D}f(n) \leq \|f\|_l\left(\frac{\card(D)}{P}\right)^{1/q},$$
where $q>1$ is a real number satisfying $\frac{1}{l}+\frac{1}{q}=1.$
Hence we have $\card(D) \geq (c_1^{-1}\beta)^{l/(l-1)}P$ as claimed.
\end{proof}
\begin{prop}\label{ub-lthmoment}
With the notations as in Theorem~\ref{3-term-modp}, we have
\begin{equation}\label{ub-leq}
\|f_C*f_B\|_l \leq \frac{c}{\beta},
\end{equation}
where $c>0$ is a constant depending only upon $t$ and $l.$
\end{prop} 
For proving Proposition~\ref{ub-lthmoment}, we first observe the following equality which is easy to verify:

\begin{equation}\label{lth1eq}
\|f_C*f_B\|_l^l= \frac{1}{P\card(B)^ld(C)^l}\sum_{y_i\in B}\card\left((C-y_1)\cap \cdots \cap (C-y_l)\right). 
\end{equation}
Given $\wt{y}=(y_1\cdots y_l) \in B^l$, let $G(\wt{y})$ be the graph with vertex set equal to $\{y_1\cdots y_l\}$ and $y_i$ is joined by an edge to $y_j$ if and only if $y_i-y_j\in S'$, where $S'$ is as in Theorem~\ref{3-term-modp}. Let $C(G(\wt{y}))$ denotes the number of connected components of $G(\wt{y}).$
Given $G(\wt{y})$ with $C(G(\wt{y}))= r$, let $D(G(\wt{y}))$ be a subset of $\{1,\cdots,l\}$ with $\card(D(G(\wt{y}))) =r$ and for $i,j \in D(G(\wt{y}))$ with $i \ne j$, we have $y_i$ and $y_j$ belongs to different connected components of $G(\wt{y}).$

\begin{lem}\label{bintersec-p}
Let $(y_1\cdots y_l)\in B^l$ with $C(G(y_1\cdots y_l))=r$. Then we have
$$\Card\left((C-y_1)\cap \cdots \cap (C-y_l)\right) \leq \frac{c(r)Pd(C)^r}{\beta^r},$$
where $c(r)>0$ is a constant depending only upon $r.$
\end{lem}
\begin{proof}
We have
$$\Card\left((C-y_1)\cap \cdots \cap (C-y_l)\right) \leq \Card\left(\cap_{j\in D(G(\wt{y}_l))}(C-y_j)\right).$$
We have  $\card(D(G(\wt{y}_l))) =r$ and for $i,j \in D(G(\wt{y}_l))$ with $i\ne j$, the element $y_i-y_j$ does not belong to  $S'$. Therefore the result follows using~\eqref{intersecp}.  
\end{proof}
The following lemma is easy to verify.
\begin{lem}
Let $\wt{y} \in B^l.$
If $y_i$ and $y_j$ belongs to the same connected components of $G(\wt{y})$, then $y_i-y_j\in lS'$.
\end{lem}
Using this we prove the following lemma.
\begin{lem}\label{noB}
The number of $\wt{y}_l\in B^l$ with $C(G(\wt{y}_l))= r$ is at-most $c(t,r)\left(\card(B)\right)^r$ where $c(t,r)>0$ is a constant depending only upon $r$ and $t.$ We may take $c(t,r) =\binom{l}{r} (rt)^{2l^2}.$

\end{lem}
\begin{proof}
Let $J$ be a subset of $\{1,\cdots,l\}$ with $\card(J)= r.$ First we obtain an upper bound for the number of $\wt{y}_l \in B^l$ such that $D(G(\wt{y}_l))= J.$ For this we note that for any $i \in \{1,\cdots,l\}\setminus J$, there exists some $j \in J$ such that $y_i-y_j \in lS'.$ Hence the number of such $\wt{y}_l \in B^l$ is at-most $\left(r\card(lS')\right)^{l-r}\card(B)^r.$ Since there are $\binom{l}{r}$ many different sets $J$ possible, the lemma follows.

\end{proof}

\begin{proof}[Proof of Proposition~\ref{ub-lthmoment}] Using~\eqref{lth1eq}, it follows that
$$\|f_C*f_B\|_l^l = \sum_{r=1}^l\frac{1}{P\card(B)^ld(C)^l} \sum_{\wt{y}_l \in B^l, C(G(\wt{y}_l))=r}\Card(\cap_{i=1}^l(C-y_i)).$$
Using this and Lemmas~\ref{bintersec-p} and~\ref{noB}, we obtain that
$$\|f_C*f_B\|_l^l \leq \sum_{r=1}^l\frac{1}{\card(B)^{l-r}d(C)^{l-r}\beta^r}c(r)c(t,r),$$
where $c(r)$ is as in Lemma~\ref{bintersec-p} and $c(t,r)$ is as in Lemma~\ref{noB}. Since from assumption we have $\card(B)d(C)\geq 1$ and $\beta \leq 1$, the result follows with $c= l\, {\rm max}_r c(r)c(t,r).$

\end{proof}
The claim (i) in  Theorem~\ref{3-term-modp} is an immediate consequence of Propositions~\ref{lbD} and~\ref{ub-lthmoment}. The claim (ii) in Theorem~\ref{3-term-modp} follows using this and Lemma~\ref{lblambdaL}. 

\section{Proof of Proposition~\ref{intersec}}
We shall deduce Proposition~\ref{intersec} as an easy corollary of the following result.
\begin{thm}\label{cselberg}
Let $N'$ be a natural number and $G(X) = \prod_{i=1}^m (e_iX+d_i) \in\Z[X]$ be a 
polynomial with $e_i, d_i \in \Z$ and $|e_i|+|d_i| \leq c_1N'^{100}.$ If $\Delta(G): = \prod_ia_i\prod_{i\ne j}(e_id_j-e_jd_i) \ne 0$, then for any $c< 1$, we have
\begin{equation}
\Card\{n \leq N': {\rm gcd}\left(G(n), P(N'^{c})\right)= 1\} \leq c_2 \frac{N'\log^m\log N'}{\log^m N'},
\end{equation}
where $c_2 =c_2(m,c_1) >0$ is a constant depending only upon $m$ and $c_1$ and in particular does not depend upon $N'.$
\end{thm}

\begin{proof}[Proof of Proposition~\ref{intersec}]
Recall that with $G(X) = F(b+MX)$, we have $$A_{b} \subset \{n \leq N/M: {\rm gcd}(G(n),P(N^{1/(4k+1)})=1\}.$$
Using this, it follows that
$$\cap_i(A_b+h_i-h_1) \subset \{n \leq N/M: {\rm gcd}(H(n),P(N^{1/(4k+1)})=1\},$$
where $H(X) = \prod_{j=1}^r F'(X+h_1-h_i).$ The assumption that $h_i-h_j \notin (S-S)$ implies that the discriminant of $G$ is non-zero. 
Using Theorem~\ref{cselberg} with $N' = \frac{N}{M}$ and $G$ being the polynomial as above, we obtain that
$$\card(\cap_i(A_b+h_i-h_1)) \leq c_2 \frac{N\log^{kr}\log N}{M\log^{kr} N},$$
where $c_2$ is a constant depending only upon $l$ and $F.$
The result follows using this and the observation that $\card(\cap_i(A_b+h_i))= \card(\cap_i(A_b+h_i-h_1)).$
\end{proof}

Let $G \in \Z[X]$ be a polynomial of degree $m$. For any prime $p$, let $\nu_p$ 
denotes the number of $x\in \Z/p\Z$ such that $G(x)\equiv 0\p.$ For any prime $p$ and integer $n$, we set $g(p)= \frac{\nu_p}{p}.$ 
Then it is easy to verify that for any real numbers $1 \leq w \leq z$, we have

\begin{equation}
\prod_{w \leq p \leq z}\left(1-g(p)\right)^{-1} \leq K\left(\frac{\log z}{\log w}\right)^m,
\end{equation}
where $K$ is an absolute constant. We also have
$$ \sum_{n \leq x, G(n) \equiv 0 \pmod{d}}1 = xg(d) +r(d),$$
with $|r(d)| \leq g(d) d.$
Then we have the following result

\begin{thm}\cite[Theorem 6.9, page number 69]{FI}
\label{beta} Let $z\geq 2$ and $D \geq z^{9m+1}.$
Then we have 
\begin{equation}\label{bb}
\Card\{n \leq x: \text{gcd}(G(n), P(z))= 1\} \leq \left(1+K^{10}e^{9m-s}\right) x\prod_{p\leq z}\left(1-g(p)\right)+\sum_{d\leq D}|r(d)|,
\end{equation}
where $s= \log D/\log z.$
\end{thm}

\begin{proof}[Proof of Theorem~\ref{cselberg}] We have 
$$\sum_{d\leq D}|r(d)| \leq \sum_{d\leq D}g(d)d \leq D\prod_{p\leq D}(1+g(p)) \ll D\log^mD.$$
Now $g(p) =
\frac{m}{p}$ for all $p$ not dividing $\Delta(G)$. From the
assumption, we have $\Delta(G)= \prod_{i=1}^me_i\prod_{i\ne
j}(e_id_j-e_jd_i) \leq c_1^{200} N'^{200m}.$ Therefore the number of primes dividing
$\Delta(G)$ is at-most $c(m,c_1)\log N'$, where $c(m,c_1)$ is a constant depending only upon $m$ and $c_1.$ Hence 
\begin{equation}
\prod_{p\leq z}(1-g(p)) \leq \prod_{p\leq z}(1-\frac{m}{p})\prod_{p\leq c(m,c_1)\log N'}(1-\frac{m}{p})^{-1} \leq c(m,c_1)\frac{\log^m \log N'}{\log^mz}.
\end{equation}
Therefore using Theorem~\ref{beta} with $D= \frac{N'}{\log^{2m+1}N'}$
 and $z = D^{\frac{1}{9m+1}}$ , we obtain the result if $c < \frac{1}{9m+1}.$ The result for larger $c$ follows using this and observing that $\Card\{n \leq N': \text{gcd}(G(n), P(z))= 1\}$ is a decreasing function of $z$.
\end{proof}

\section{Proof of Proposition~\ref{restriction-cor}}
Let $G(X) = \prod_{i=1}^m(e_iX+d_i)$ be a polynomial with $e_i, d_i \in \Z$ and $\Delta(G) \ne 0$. Moreover we shall assume that $G$ is non-degenerate.
The following result is a rewording of \cite[Proposition 4.2]{GT-enveloping}. 
\begin{prop}
Let $R, N$ be large numbers such that $1\ll R \ll N^{1/10}$ and let $G$ be a polynomial as above. Let $h: \Z/N\Z \to \mathbb{C}$ be a function satisfying the following:
\begin{equation}
h(n) \ne 0 \implies {\text gcd}(G(n), P(R)) =1,
\end{equation}
where $n \in [1, N].$ Then for any  real number $l>2$, we have
\begin{equation}
\left(\sum_{t \in \Z/N\Z}|\widehat{h}(t)|^l\right)^{2/l} \leq c(l,m) \frac{1}{\log^mR}\prod_p(1-\frac{1}{p})^{-m}(1-g(p))\frac{1}{N}\sum_{n}|h(n)|^2,
\end{equation}
where $c(l,m)>0$ is a constant depending only upon $l$ and $m.$
\end{prop}
Applying this result with $G(X) = F(b_0+MX)$, $N=P$ and $h= f_{A'}$, we obtain 
\begin{cor} Given $c_1, \cdots c_s\in (\Z/P\Z)^*$ and $m_1, m_2, \cdots m_s$ with $\sum_{i}m_i >2$, 
\begin{equation}\label{ft}
\sum_{t}\prod_{i=1}^s |\widehat{f}_{A'}(c_it)|^{m_i} \leq \frac{c(F,\sum_i{m_i})}{\delta^{\sum_im_i}},
\end{equation}
where $c(F,l)>0$ is a constant depending only on $F$ and $l.$
\end{cor}

\begin{proof}[Proof of Proposition~\ref{restriction-cor}]
We have 

\begin{equation}\label{diff1}
|\Lambda_L(f_{A'}) - \Lambda_L(f_{A'}*B)| = P^{s-1}\left|\sum_t\prod_{i=1}^s\widehat{f}_{A'}(c_ic_1^{-1}t)\left(1-\prod_{i=1}^s\widehat{f}_{B}(c_ic_1^{-1}t)\right)\right|.
\end{equation}
Since for $t\in \bigcup_ic_ic_1^{-1}.S_{\epsilon_1}$, we have $|\widehat{f}_B(t)-1|\leq \epsilon_2$, it follows that for $t \in S_{\epsilon_1}$, we have $|1-\prod_i\widehat{f}_{B}(c_ic_1^{-1}t)| \ll \epsilon_2.$. Hence using this and~\eqref{ft} we have

\begin{equation}\label{diff2}
\sum_{t \in S_{\epsilon_1}}\left|\prod_{i=1}^s\widehat{f}_{A'}(c_ic_1^{-1}t)\left(1-\prod_{i=1}^s\widehat{f}_{B}(c_ic_1^{-1}t)\right)\right| \ll \epsilon_2\sum_t\left|\prod_{i=1}^s\widehat{f}_{A'}(c_ic_1^{-1}t)\right| \leq \epsilon_2 \frac{c(F,L)}{\delta^s}.
\end{equation}

For $t\notin LS_{\epsilon_1}$, we have $|f_{A'}(t)|^{1/2}\leq \epsilon_1{1/2}$. Therefore the contribution in right hand side of~\eqref{diff1} coming from such $t$ is at-most and  hence we using~\eqref{ft}, we have
\begin{equation}\label{diff3}
\epsilon_1^{1/2}\sum_{t\notin S_{\epsilon_1}}\left|\widehat{f}(t)^{1/2}\prod_{i=2}^s\widehat{f}_{A'}(c_ic_1^{-1}t)\right|\ll \epsilon_1^{1/2}\frac{c(F.L)}{\delta^s}.
\end{equation}
Using~\eqref{diff1}, \eqref{diff2} and \eqref{diff3}, the result follows.

\end{proof}

\section{Relation between $g_L$ and $h_L$}
When the number of variables $s$ in a translation invariant linear equaltion $L$ is $3$, a relation between $g_L$ and $h_L$ follows from the following result.
\begin{thm}[Varnavides theorem]\label{VV}
Let $L$ be a translation invariant linear equation and $g_L, g_L^*$ are functions as defined in Definition~\ref{glhl}. Let $\eta>0$ and $D \subset \Z/P\Z$ with $\card(D)\geq \eta P$. Then the number of solution of $L$  in $D$ is at least $$\frac{\eta}{2}\frac{P(P-1)}{\left(g^{*-1}(\eta/2)\right)^2}.$$
\end{thm}
\begin{proof}
For the brevity of notation, we write $t$ to denote $g^{*-1}(\eta/2).$ Since the assumption implies that $D$ is non empty and hence contains at least one trivial solution of $L$, the result is true if $t\geq P.$ Hence we may assume that $t < P.$

\vspace{1mm}
\noindent
Given any $a\in \Z/P\Z$ and $d\in \Z/P\Z \setminus \{0\}$, let
$I_{a,d}: = \{a+d,\cdots, a+td\}$ be an arithmetic progression of length $t$. We say that $I_{a,d}$ is a ``good'' progression, if $\card(D\cap I_{a,d})\geq \frac{\eta}{2}t.$ We claim that if $I_{a,d}$ is good then $D' = D\cap I_{a,d}$ contains a non-trivial three term arithmetic progression. For this we first notice that since $P$ is prime and $d$ is a non zero element of $\Z/P\Z$, we have $\card(D') =\card(\frac{D'-a}{d}).$ Hence $\frac{D'-a}{d} \subset [1,t]$ and contains at least $\frac{\eta}{2}t$ elements.
Therefore using the properties of $g_L$ and definition of $g^{*-1}_L(\eta/2)$, it follows that $\frac{D'-a}{d}$ contains a non-trivial solution of $L$, which proves the claim.
 Now we shall obtain a lower bound for the number of good $I_{a,d}.$

\vspace{1mm}
\noindent
Now for any fixed $d_0$, we have the following identity:
$$\sum_{a\in \Z/P\Z}\card(D\cap I_{a,d_0})= t \card(D).$$
This follows by observing that any $c \in D$ belongs to exactly $t$ many $I_{a,d_0}$. From the above identity it follows that for any fixed $d_0$, the number of good $I_{a,d_0}$ is at least $\frac{\card(D)}{2}$ which by assumption is at-least $\frac{\eta}{2}P.$ Now varying $d_0$, we obtain that the number of good $I_{a,d}$ is at least $\frac{\eta}{2}P(P-1).$ The lemma follows using this and the observation that a given non-trivial solution of $L$ can belong to at most $t^2$ many good $I_{a,d}.$
\end{proof}
Using Theorem~\ref{VV}, we immediately obtain the following result.
\begin{cor}
Let $L$ be a translation invariant equation in $s$ many variables and $g_L$ be a function as satisfying the properties as in Definition~\ref{glhl}. When $s=3$, then $h_L(\eta) = \frac{\eta}{\left(2g^{*-1}_L(\eta/2)\right)^2}$ is a function satisfying the properties as in Definition~\ref{glhl}~(i).
\end{cor}
As remarked earlier, Thomas Bloom~\cite{Bloom} showed that there exists an absolute constant $c>0$ depending only on $L$ such that the function $g_L(N)= c\left(\frac{\log^5\log N}{\log N}\right)^{s-2}$ satisfies the above properties. In this case $g^{*-1}_L(\eta) \leq \exp(c_1\eta^{-1/(s-2)}\log^6 \log (\frac{1}{\eta}))$ with $c_1>0$ being a constant depending only upon $L.$ Therefore when $s=3$, there exists an absolute constant $c>0$ such that we may take  
\begin{equation}\label{hL3}
h_L(\eta)= \exp\left(-c\eta^{-1}\log^6\log\frac{1}{\eta}\right).
\end{equation}

\section{Proof of Theorem~\ref{mainthm}}
Let $S$ be as in Proposition~\ref{restriction-cor} and $B \subset \Z/P\Z$ defined as  $$B= {\rm Bohr}(S,\epsilon_2):=\{x\in \Z/P\Z: \left|\exp\left(\frac{2\pi ixt}{P}\right)-1\right| \leq \epsilon_2 \forall t\in S\}.$$
Then $B$ satisfies the assumptions in Proposition~\ref{restriction-cor}. We shall choose $\epsilon_1$ and $\epsilon_2$ in such a way that $B$ also satisfies the assumptions in Proposition~\ref{prop-3-conv}.

\begin{lem}[Lemma 4.20 \cite{TV}] Given any set $C\subset \Z/P\Z$ and any real number $\epsilon >0$, we have
$$\card({\rm Bohr}(C,\epsilon)) \geq (\epsilon)^{|C|}P.$$
\end{lem}
Moreover an immediate consequence of~\eqref{ft} is the following upper bound for the cardinality of $S$:
$$\Card(S) \leq \frac{\epsilon_1^{-3}c(F,L)}{\delta^3}.$$
Therefore we have $\card(B) \geq \log^{k+101}P$ and hence $B$ satisfies the assumption of Proposition~\ref{prop-3-conv} provided, we have
\begin{equation}\label{cepsilons}
\frac{\epsilon_1^{-3}c(F,L)}{\delta^3}\log(\epsilon_2)\geq -\frac{\log P}{2}
\end{equation}
and $P$ is sufficiently large.  Therefore if~\eqref{cepsilons} is saisfied, then using Propositions~\ref{prop-3-conv} and~\ref{restriction-cor}, we have
$$\Lambda_L(f_{A'})\geq c_1h_L(c_2\delta^{l/(l-1)})P^{s-1}-c(F,L)\frac{\epsilon_2+\epsilon_1^{0.5}}{\delta^s}P^{s-1}.$$
Therefore choosing 
\begin{equation}
\epsilon_2 = \epsilon_1^{0.5}= \frac{\delta^s c_1h_L(c_2\delta^{l/(l-1)})}{c(F,L)},
\end{equation}
we obtain 
\begin{equation}
\Lambda_L(f_{A'})\geq c_1h_L(c_2\delta^{l/(l-1)})P^{s-1},
\end{equation}
where $c_1$ and $c_2$ are constants depending only upon $F$ and the linear equation $L$, provided our choice of $\epsilon_1$ and $\epsilon_2$ satisfies~\eqref{cepsilons}. Since $s=3$, with the choice of $h_L$  provided by~\eqref{hL3}, we have that for some $c_1,c_2>0$, we have
$$\epsilon_1 = \exp\left(-c_1\delta^{-l/(l-1)}\log^6\log\frac{1}{\delta}\right),\; \; \epsilon_2 =\exp\left(-c_1\delta^{-l/(l-1)}\log^6\log\frac{1}{\delta}\right).$$
Therefore~\eqref{cepsilons} holds using the assumed lower bound for $\delta$, provided $l$ is choosen sufficiently large depending on $\epsilon$, where $\epsilon$ is as in Theorem~\ref{mainthm}.
\end{document}